\documentclass[12pt]{amsart}

\usepackage[margin=2cm]{geometry}
\usepackage[french,english]{babel}
\usepackage{amsmath}
\usepackage{amsthm}
\usepackage{amssymb}
\usepackage{enumerate}
\usepackage{ucs}
\usepackage[utf8]{inputenc}
\usepackage[T1]{fontenc}
\usepackage{color}
\usepackage{ stmaryrd }
\usepackage{hyperref}
\usepackage[retainorgcmds]{IEEEtrantools}
\usepackage[initials, shortalphabetic]{amsrefs}

  \newcommand{\Lang}{\mathcal L}

  \newcommand{\I}{[0,1]}

 \newcommand{\bit}{\begin{itemize}}
 \newcommand{\eit}{\end{itemize}}
  \newcommand{\ben}{\begin{enumerate}}
 \newcommand{\een}{\end{enumerate}}
 \newcommand{\eps}{\epsilon}

 \theoremstyle{definition}
\newtheorem{thm}{Theorem}
\newtheorem{cor}[thm]{Corollary}
\newtheorem{lem}[thm]{Lemma}
\newtheorem{prop}[thm]{Proposition}
\newtheorem{df}[thm]{Definition}

\newtheorem{ex}[thm]{Example} 

\newtheorem{rmq}[thm]{Remark}
\newtheorem*{rmq*}{Remark}

\newtheorem*{notation}{Notation}

  \title{Reconstruction of separably categorical metric structures}
  \author{Ita\"{i} Ben Yaacov and Adriane Ka\"{i}chouh}

\begin{document}
\reversemarginpar
\maketitle

\begin{abstract}
We extend Ahlbrandt and Ziegler's reconstruction results (\cite{MR831437}) to the metric setting: we show that separably categorical structures are determined, up to bi-interpretability, by their automorphism groups.
\end{abstract}

\section*{Introduction}

Categoricity offers an ideal setting for reconstruction: a lot of information on a categorical structure can be recovered from the action of its automorphism group. Indeed, the Ryll-Nardzewski theorem asserts that a classical countable structure is $\aleph_0$-categorical if and only if its automorphism group acts oligomorphically on it. From this, definability in countably categorical structures translates as invariance under the action of the automorphism group.

An analogous phenomenon occurs in the continuous setting: a metric structure is separably categorical if and only if the action of its automorphism group is approximately oligomorphic (\cite[theorem 12.10]{MR2436146}). This continuous Ryll-Nardzewski theorem again implies that definability boils down to invariance by automorphisms (see section 1).

In this paper, we focus on a reconstruction result due to Ahlbrandt and Ziegler (\cite{MR831437}) which states that countably categorical structures are determined, up to bi-interpretability, by their automorphism groups (regarded as topological groups). 
We extend Ahlbrandt and Ziegler's result to the continuous setting. More precisely, we introduce the notion of an interpretation between metric structures and prove that two separably categorical structures are bi-interpretable if and only if their automorphism groups are topologically isomorphic. 

This guarantees that every model-theoretic property of separably categorical structures will translate into a topological property of their automorphism groups. Tsankov and the first author (\cite{itaitodor}), and then Ibarluc\'{i}a (\cite{tomas-hierarchy}), are precisely studying model-theoretic properties directly on groups.

Although our result encompasses its classical counterpart, the proof we give is fundamentally metric and is quite different from the original one. Indeed, we apply a construction of Melleray (\cite[theorem 6]{MR2767973}) that provides a canonical way to make a metric structure out of any Polish group (we will call this the \textit{hat structure} associated to the group; see subsection 3.3 for a definition). The heart of the reconstruction consists in showing that every separably categorical metric structure is in fact bi-interpretable with the hat structure of its automorphism group.

\section{Definability}

In this section, we prove the aforementioned fact that in separably categorical structures, definability amounts to invariance under the action of the automorphism group. 
First, we introduce the following item of notation.

\begin{notation}
If $\rho$ is a bounded pseudometric on a structure $M$, then $(M,\rho)$ will denote the quotient metric space induced by $\rho$. For such a $\rho$, let $\rho^\omega$ be the pseudometric on $M^\omega$ defined by
$$\rho^\omega(a,a') = \sum_{n < \omega} \frac{1}{2^n} \rho(a_n,a'_n).$$
When $\rho$ is a metric, so is $\rho^\omega$, which then induces the product topology on $M^\omega$.
\end{notation}

\begin{prop}\label{criterion}
Let $M$ be a separably categorical metric structure and $G$ its automorphism group. Let $P : M^n \to \I$ a continuous predicate on $M$. 
Then $P$ is definable in $M$ if and only if $P$ is $G$-invariant.
\end{prop}

\begin{proof}
$\Rightarrow$] If $P$ is definable, there is a sequence $(\varphi_k)_{k \geqslant 1}$ of formulas which converges uniformly to $P$. Now $G$ preserves (interpretations of) formulas so $P$ is also $G$-invariant.

$\Leftarrow$] Suppose that $P$ is $G$-invariant. If $a$ and $b$ have the same type in $M^n$, then, since $M$ is approximately homogeneous (\cite[corollary 12.11]{MR2436146}) and $P$ is continuous, the $G$-invariance of $P$ gives that $P(a) = P(b)$.

Thus, $P$ induces a metrically continuous map $\Phi : S_n(T) \to \I$ on types, defined by $\Phi(p) = P(a)$ for $a \in M^n$ of type $p$. Since every type is realized in $M$ (by the Ryll-Nardzewski theorem \cite[fact 1.14]{MR2291717}), the map $\Phi$ is well-defined. 

Now, by the Ryll-Nardzewski theorem again, the logic topology and the $d$-metric topology on $S_n(T)$ coincide. This implies that $\Phi$ is continuous for the logic topology as well. Thus, by theorem 9.9 of \cite{MR2436146}, the predicate $P$ is definable.
\end{proof}

\begin{rmq}
The same holds for predicates in an infinite number of variables. In fact, if $M^\omega$ is endowed with $d^\omega$, then the Ryll-Nardzewski theorem can be reformulated as follows: a metric structure $M$ is separably categorical if and only if the space $(S_\omega(M), d^\omega)$ of types in infinitely many variables is compact. Thus, the proof above readily adapts to an infinite number of variables.
\end{rmq}

\section{Reconstruction up to interdefinability}

We begin by reconstructing separably categorical structures up to interdefinability, mirroring Ahlbrandt and Ziegler's theorem 1.1 (in \cite{MR831437}). The proof is exactly the same as in the discrete setting.

\begin{df}
Let $M$ and $N$ be two structures on the same universe. We say that $M$ and $N$ are \textbf{interdefinable} if they have the same definable relations.
\end{df}

\begin{prop}
Let $M$ and $N$ be two separably categorical metric structures on the same universe, in languages $\Lang_M$ and $\Lang_N$ respectively. Then $M$ and $N$ are interdefinable if and only if their automorphism groups are equal.
\end{prop}

\begin{proof}
$\Rightarrow$] Assume that $M$ and $N$ are interdefinable and let $R$ be a relation in $\Lang_N$. Since it is definable in $N$, it is definable in $M$ as well, so proposition \ref{criterion} implies that it is $\text{Aut}(M)$-invariant. Thus $\text{Aut}(M)$ preserves every relation in $\Lang_N$ so $\text{Aut}(M) \subseteq \text{Aut}(N)$. Similarly, we obtain that $\text{Aut}(N) \subseteq \text{Aut}(M)$.

$\Leftarrow$] Conversely, assume that $\text{Aut}(M) = \text{Aut}(N)$ and let $R$ be definable in $M$. Then, by proposition \ref{criterion}, it is $\text{Aut}(M)$-invariant hence $\text{Aut}(N)$-invariant by assumption. Thus, $R$ is definable in $N$ and the two structures have the same definable relations.
\end{proof}

\section{Reconstruction up to bi-interpretability}

\subsection{Interpretations}

In the classical setting, an interpretation of a structure $M$ in an other structure $N$ is an embedding of $M$ into a definable quotient of a finite power of $N$, that is, into the imaginaries of $N$. As Usvyatsov and the first author pointed out in \cite{MR2657678}, the right definition of imaginaries in metric structures should allow classes of infinite tuples and this is also true for interpretations.


\begin{df}
Let $M$ and $N$ be two metric structures in languages $\Lang_M$ and $\Lang_N$ respectively. An \textbf{interpretation} of $M$ in $N$ consists of the following data:
\bit
\item a definable pseudometric $\rho$ on $N^\omega$ and
\item an isometric map $\varphi : (M,d_M) \to \widehat{(N^\omega,\rho)}$
\eit
such that
\bit
\item the predicate $P: N^\omega \to \I$ defined by $P(x) = \rho(x,\varphi(M))$ is definable in $N$ and 
\item for every formula $F$ in $\Lang_M$, the formula $P_F : \varphi(M)^r \to \I$ defined by $P_F(x) = F(\varphi^{-1}(x))$ is definable in $N$.
\eit
\end{df}

To verify the last condition, it suffices to check it on relation symbols in $\Lang_M$ and on predicates of the form $(x,y) \mapsto d(x,f(y))$, where $f$ is a function symbol in $\Lang_M$.

\begin{rmq}
If $M$ and $N$ are classical structures, they can be made into discrete metric structures. Then every interpretation of $M$ in $N$ (in the metric sense, as defined above) induces a classical interpretation of $M$ in $N$. To see this, given a metric interpretation $\varphi$ of $M$ in $N$, use the continuity of the associated pseudometric to choose a big enough $n$ such that the elements in the image $\varphi(M)$ (which is discrete) are determined by their restriction to the first $n$ coordinates. Then the equivalence relation on $N^n$ induced by restriction of $\rho$ is well-defined and definable, and it yields an interpretation of $M$ in $N$.
\end{rmq}

If $M$, $N$ and $K$ are metric structures, $\varphi: (M,d_M) \to \widehat{(N^\omega,\rho_N)}$ is an interpretation of $M$ in $N$ and $\psi: (N,d_N) \to \widehat{(K^\omega,\rho_K)}$ is an interpretation of $N$ in $K$, then we can \textbf{compose the interpretations} $\psi$ and $\varphi$ as follows. 

\noindent Taking the product of $\psi$, we get an isometric map $\psi^\omega : (N^\omega, d_N^\omega) \to \widehat{(K^{\omega \times \omega},\rho_K^\omega)} = \widehat{(K^\omega,\rho_K^\omega)}$. Now, since $\psi(N)$ is definable in $\widehat{(K^\omega,\rho_K)}$, the image $\psi^\omega(N^\omega)$ is definable in $\widehat{(K^\omega,\rho_K^\omega)}$ too. Besides, $\rho_N$ is a definable pseudometric, so its pushforward by $\psi^\omega$ also is. Then, by \cite[proposition 3.6]{MR2723769}, it extends to a definable pseudometric $\rho$ on $\widehat{(K^\omega,\rho_K^\omega)}$. Thus, the isometric map $\psi^\omega \circ \varphi : (M,d_M) \to \widehat{(K^\omega,\rho)}$ is an interpretation of $M$ in $K$.

\begin{rmq}\label{categorical}
If $N$ is separably categorical and $\rho$ is a definable pseudometric on $N^\omega$, then the structure $\widehat{(N^\omega,\rho)}$ is separably categorical too (and thus definability corresponds to invariance by the automorphism group). Indeed, the automorphism group of $N$ acts approximately oligomorphically on $N^\omega$ and thus on $\widehat{(N^\omega,\rho)}$ too. This implies that the whole automorphism group of $\widehat{(N^\omega,\rho)}$ is approximately oligomorphic so, by the Ryll-Nardzewski theorem, that $\widehat{(N^\omega,\rho)}$ is separably categorical.

In particular, any structure that is interpretable in a separably categorical one is itself separably categorical. That is the reason why it is necessary to impose an oligomorphicity restriction in theorems \ref{arbre} and \ref{onesidereconstruction}.
\end{rmq}

\begin{df}
Let $M$ and $N$ be two metric structures. We say that $M$ and $N$ are \textbf{bi-interpretable} if there exist interpretations $\varphi$ of $M$ in $N$ and $\psi$ of $N$ in $M$ such that $\psi \circ \varphi$ and $\varphi \circ \psi$ are definable.
\end{df}

In the rest of this section, we argue that interpretations between separably categorical structures correspond to continuous homomorphisms between their automorphism groups.

\subsection{From interpretations to group homomorphisms}

The first side of this correspondence is not too surprising, for it amounts to saying one can get information on the automorphism group from the structure. The process is however nicely functorial.

\begin{prop}
Let $M$ and $N$ be two metric structures and $\varphi$ an interpretation of $M$ in $N$. Then $\varphi$ induces a homomorphism of topological groups $\text{Aut}(\varphi)$ from $\text{Aut}(N)$ to $\text{Aut}(M)$.
\end{prop}

\begin{proof}
Let $g$ be an automorphism of $N$. Then $g$ leaves $\varphi(M)$ and the predicates $P_R$ invariant so it induces an automorphism of $(\varphi(M), (P_R))$ and thus of $M$. More formally, if $a$ is an element of $M$, we put $\text{Aut}(\varphi)(g)(a) = \varphi^{-1}(g(\varphi(a)))$. Then $\text{Aut}(\varphi)$ is the conjugation by $\varphi$ so it is a group homomorphism.

And since $\rho$ is continuous, it is easy to see that $\text{Aut}(\varphi)$ is continuous.
\end{proof}

The map $\varphi \mapsto \text{Aut}(\varphi)$ is functorial: it respects composition.

\begin{lem}
Let $M$, $N$ and $K$ be metric structures, $\varphi$ an interpretation of $M$ in $N$ and $\psi$ an interpretation of $N$ in $K$. Then $\text{Aut}(\psi \circ \varphi) = \text{Aut}(\psi) \circ \text{Aut}(\varphi)$.
\end{lem}

\begin{lem}\label{definableinterpretation}
Let $M$ be a separably categorical metric structure and $\varphi$ an interpretation of $M$ in itself. Then $\varphi$ is definable in $M$ if and only if $\text{Aut}(\varphi) = \text{id}_{\text{Aut}(M)}$.
\end{lem}

\begin{proof}
$\Rightarrow$] If $\varphi$ is definable, then $\varphi$ is $\text{Aut}(M)$-invariant. Then, if $g \in \text{Aut}(M)$ and $a \in M$, we have $\text{Aut}(\varphi)(g)(a) = \varphi^{-1}(g(\varphi(a))) = \varphi^{-1}(\varphi(g(a))) = g(a)$ and thus $\text{Aut}(\varphi)$ is the identity.

$\Leftarrow$] If $\text{Aut}(\varphi)$ is the identity, the same computation shows that $\varphi$ is $\text{Aut}(M)$-invariant. Since $\varphi$ is continuous (it is isometric), this implies that $\varphi$ is definable (by proposition \ref{criterion}). 
\end{proof}

This lemma will yield the first direction of theorem \ref{reconstruction}.

\subsection{A special structure: a group with a hat}

We now proceed to the second part of the correspondence: the actual reconstruction. To this aim, we come down to a canonical structure built from the automorphism group and with which the structure is bi-interpretable.The following construction is due to Melleray (\cite[theorem 6]{MR2767973}).

Let $M$ be a metric structure and $G$ be its automorphism group. Whenever we endow $G$ with a compatible left-invariant metric $d_L$, we can consider the structure $\widehat G$ whose universe is the left completion $\widehat G_L$ of $(G, d_L)$ and whose relations are all those maps of the form $R_C(x) = d(x, C)$, for some orbit closure $C$ of $(\widehat G_L)^n$ under the diagonal action of $G$. Then the automorphism group of $\widehat G$ is $G$. 

 Now fix a dense sequence $\xi \in M^\omega$. Then, the metric on $G$ given by $d_\xi(g,h) = d(g\xi, h\xi)$ is a compatible left-invariant metric.

\begin{prop}\label{ben, oui}
The structure $\widehat G$ (obtained from this particular metric $d_\xi$) is interpretable in $M$.
\end{prop}

\begin{proof}
Consider the map $\psi : g \mapsto g\xi$ from $(G,d_\xi)$ to $\overline{G \cdot \xi} \subseteq M^\omega$. It is isometric so it extends to the left completion of $G$. Then $\psi$ is an interpretation of $\widehat G$ in $M$. 

Indeed, the predicate $P : x \mapsto d(x,\psi(\widehat G)) = d(x,\overline{G \cdot \xi})$ on $M^\omega$ is $G$-invariant so it is definable in $M$ by proposition \ref{criterion}. Moreover, if $C$ is an orbit closure and $R = R_C$ is the associated predicate in $\widehat G$, we have
\begin{align*}
P_R(gx) 
&= R(\psi^{-1}(gx)) \\
&= R(g\psi^{-1}(x)) \\
&= R(\psi^{-1}(x))  \text{       because $R$ is invariant by the automorphism group} \\
&= P_R(x),
\end{align*}
so $P_R$ is definable, which completes the proof.
\end{proof}

\begin{rmq}\label{identification}
In fact, since the image of $\psi$ is dense, $\overline{G \cdot \xi}$ is exactly the left completion of $G$ and from now on, we identify $\widehat G$ with $\overline{G \cdot \xi}$.
\end{rmq}

The above proposition, along with remark \ref{categorical}, implies that if $M$ is separably categorical, then so is $\widehat G$. And in that case, if $d_L$ is any other compatible left-invariant metric, then the associated hat structure is bi-interpretable with $\widehat G$: the two metrics generate the same topology so they are continuous with respect to each other, and their left-invariance implies, by proposition \ref{criterion}, that they are definable from each other. All the hat structures obtained from $G$ are bi-interpretable and we will therefore identify them.

Moreover, if $M$ is separably categorical, then the structure $M$ is also interpretable in $\widehat G$. In fact, we have the following more general result which will be the key ingredient in the proof of theorem \ref{onesidereconstruction}.

\begin{thm}\label{arbre}
Let $N$ be a metric structure and let $H$ be a subgroup of $\text{Aut}(N)$ which acts approximately oligomorphically on $N$. Then $N$ is interpretable in $\widehat H$. 
\end{thm}

\begin{proof}
Let $\zeta$ be a dense sequence in $N$. Then $\widehat H = \overline{H \cdot \zeta}$. Now the assumption ensures that the space $N \sslash H$ of orbit closures of $N$ by $H$ is compact. 

The intuition for the proof is to say that $N$ is not far from being the product $\widehat H \times N \sslash H$ and moreover that compact spaces should be interpreted in every structure. As a matter of fact, we will build a particular system of representatives of $N \sslash H$ that $\widehat H$ will interpret.

We begin by building a tree $T$ representing this compact quotient $N \sslash H$. For this, we will choose representatives, within $\zeta$, of a dense sequence of orbit closures that witnesses the compactness of this quotient, and $T$ will be the tree of their indices in $\zeta$.
More precisely, we build the tree by induction: first, there exist $\zeta_{n_1}, ..., \zeta_{n_k}$ in $\zeta$ such that the balls of radius $\frac 12$ centered in the closures of the orbits of $\zeta_{n_1},..., \zeta_{n_k}$ cover all of the quotient. The \textit{indices} of those elements constitute the first level of our tree. For the next step, we cover each of the balls $B(\zeta_{n_i}, \frac 12)$ in $N$ with a finite number of balls of radius $\frac 14$ centered in elements of $\zeta$ so that the second level of our tree consists of the indices of those centers, and so on. 

The construction ensures that for every infinite branch of $T$, the sequence $(\zeta_{\sigma(i)})$ converges in $N$. Moreover, every orbit closure corresponds to an infinite branch of $T$ (maybe even several): for every $a$ in $N$, there exists an infinite branch $\sigma$ of $T$ such that the limit of the sequence $(\zeta_{\sigma(i)})$ is in the closure of the orbit of $a$. Let $[T]$ be the set of infinite branches of $T$.

We now embed $N$ isometrically into (the completion of) a quotient of $\overline{H \cdot \zeta} \times [T]$, which we identify with $\widehat H \times [T]$. This will give the base map for our interpretation.

Endow the set $\overline{H \cdot \zeta} \times [T]$ with the following pseudometric
$$\rho((x,\sigma), (y, \tau)) = \lim_{i \to \infty} d(x_{\sigma(i)}, y_{\tau(i)}).$$
Since for every branch $\sigma$ in $[T]$, the sequence $(\zeta_{\sigma(i)})$ converges, this is also true of every $(x_{\sigma(i)})$ with $x$ in $\overline{H \cdot \zeta}$, so $\rho$ is well-defined.

We now define a map $\varphi : (\overline{H \cdot \zeta} \times [T], \rho) \to N$ by $\displaystyle{\varphi(x,\sigma) = \lim_{i \to \infty} x_{\sigma(i)}}$. By definition of $\rho$, the map $\varphi$ is isometric. In addition, the image of $\varphi$ is dense in $N$. Indeed, let $a$ be an element of $N$ and $\eps > 0$. There exists a branch $\sigma$ in $[T]$ such that $(\zeta_{\sigma(i)})$ converges to some $a'$ in $N$ which is in the same $H$-orbit closure as $a$, that is, there exists $h \in H$ such that $d(h(a'),a) < \eps$, so $d(\varphi(h \zeta_{\sigma(i)}), a) < \eps$, hence the density. 

Thus, the isometric map $\varphi$ can be extended to an isometry from the completion of $(\overline{H \cdot \zeta} \times [T], \rho)$ onto $N$. Then its inverse, call it $\tilde \varphi$, is the desired isometric map between $N$ and the completion of $(\widehat H \times [T], \rho)$. This was the first step in our intuition. 

In order to see $\tilde \varphi$ as an interpretation of $N$ in $\widehat H$, it remains to interpret $[T]$ in $\widehat H$, in other words, to code the branches of $T$ in a power of $\widehat H$ (that is $\overline{H \cdot \zeta}$ via the identification of the previous remark). The map $\tilde \varphi$ will then induce a map $N \to \widehat H \times \widehat H^\omega$, which will be the desired interpretation.

A branch can be coded by a sequence of zeroes and ones\footnote{There are many ways of doing so; we pick one. For instance, we may say that given a branch of $T$, we follow the levels of $T$ one by one, and we put a $1$ in our sequence when we hit an element of our branch and a $0$ otherwise.}. Then we code\footnote{There are also many ways of coding zeroes and ones in a power of $\overline{H \cdot \zeta}$. Here we go for a method which compares two sequences of a pair in a very simple way.} each bit by a pair of elements of $\overline{H \cdot \zeta}$. Consider the pseudometric on $\overline{H \cdot \zeta} \times \overline{H \cdot \zeta}$ defined by
$$\delta((x,x'),(y,y')) = \lvert d(x_0,x'_0) - d(y_0,y'_0) \rvert,$$
which compares the differences between the first coordinates of the two sequences of the pair. This is a definable pseudometric and we code the bit $0$ by the $\delta$-class of $(\zeta, \zeta)$ and the bit $1$ by the $\delta$-class of $(\zeta, h_0 \zeta)$ where $h_0$ is some element of $H$ that does not fix $\zeta_0$. Note that the code is invariant under the action of $H$. 

Finally, we identify branches of $T$ with their codes in $(\widehat H^2, \delta)^\omega$ and we transfer the pseudometric $\rho$ on $\widehat H \times [T]$ to a definable pseudometric $\tilde \rho$ on $\widehat H \times (\widehat H^2, \delta)^\omega$. Note that the elements of $(\widehat H^2, \delta)^\omega$ that code a branch of $[T]$ may not cover the whole of $(\widehat H^2, \delta)^\omega$, but \cite{MR2723769}'s proposition 3.6 allows us to extend $\tilde \rho$ to $(\widehat H^2, \delta)^\omega$ all the same. 

So we can now rewrite the map $\tilde \varphi$ as a map from $N$ to the completion of $(\widehat H \times (\widehat H^2, \delta)^\omega, \tilde \rho)$. The oligomorphicity of the action of $H$ on $N$ implies that the structure $\widehat H = \overline{H \cdot \zeta}$, whose automorphism group is $H$, is separably categorical. Since $\tilde \rho$ is invariant under the action of $H$, proposition \ref{criterion} then yields that the pseudometric $\tilde \rho$ is definable in $\widehat H$. Therefore, this new map $\tilde \varphi$ is an interpretation of $N$ in $\widehat H$. 
\end{proof}

\begin{cor}\label{bi-interpretable avec le chapeau}
If $M$ is separably categorical, then the structures $M$ and $\widehat G$ are bi-interpretable.
\end{cor}

\begin{proof}
The proposition implies in particular that $M$ is interpretable in $\widehat G$. Thus, it suffices to show that the compositions of the interpretations constructed in the previous propositions are definable. Both interpretations respected the actions of the automorphism groups so proposition \ref{criterion} and remark \ref{categorical} allow us to conclude.
\end{proof}

\subsection{Reconstruction}

We are now ready to complete the reconstruction.

\begin{thm}\label{onesidereconstruction}
Let $M$ and $N$ be two metric structures, with $M$ separably categorical. Let $f : \text{Aut}(M) \to \text{Aut}(N)$ be a continuous group homomorphism whose image acts approximately oligomorphically on $N$. Then $N$ is interpretable in $M$.
\end{thm}

\begin{proof}
Set $G = \text{Aut}(M)$ and $H = f(G)$. Since $H$ acts approximately oligomorphically on $N$, theorem  \ref{arbre} implies that $N$ is interpretable in $\widehat H$. And by proposition \ref{ben, oui}, the structure $\widehat G$ is interpretable in $M$. It then suffices to show that $\widehat H$ is interpretable in $\widehat G$. 

Now $H$ is the quotient of $G$ by the closed normal subgroup $\text{Ker}(f)$. If $d_L$ is a left-invariant metric on $G$, then we can endow $G$ with the following left-invariant pseudometric 
$$d'_L(g_1, g_2) = \inf \{ d_L(g_1 k_1, g_2 k_2) : k_1, k_2 \in \text{Ker}(f) \}.$$
Since $\text{Ker}(f)$ is normal, this indeed defines a pseudometric, which induces a compatible metric on $H$. Then $\widehat{(H, d'_L)}$, which we identify with $\widehat H$ (see subsection 3.3), is the quotient\footnote{Here, we do not even need to go to a power of $\widehat G$ to interpret $\widehat H$.} of $\widehat G$ by the definable pseudometric $d'_L$ and is thus interpretable in $\widehat G$. 
\end{proof}

\begin{thm}\label{reconstruction}
Let $M$ and $N$ be separably categorical metric structures. Then $M$ and $N$ are bi-interpretable if and only if their automorphism groups are isomorphic as topological groups.
\end{thm}

\begin{proof}
$\Rightarrow$] Assume that $\varphi$ and $\psi$ are interpretations that witness the bi-interpretability of $M$ and $N$. Then lemma \ref{definableinterpretation} implies that $\text{Aut}(\varphi \circ \psi) = \text{id}_{\text{Aut}(N)}$ and $\text{Aut}(\psi \circ \varphi) = \text{id}_{\text{Aut}(M)}$. But $\text{Aut}(\varphi \circ \psi) = \text{Aut}(\varphi) \circ \text{Aut}(\psi)$ so $\text{Aut}(\varphi) = \text{Aut}(\psi)^{-1}$ and $\text{Aut}(\psi)$ is an isomorphism of topological groups between $\text{Aut}(M)$ and $\text{Aut}(N)$. Note that for this direction, we do not need the categoricity of the structures.

$\Leftarrow$] By corollary \ref{bi-interpretable avec le chapeau}, $M$ is bi-interpretable with $\widehat{\text{Aut}(M)}$ and $N$ with $\widehat{\text{Aut}(N)}$. Now if the two groups are isomorphic as topological groups, then their associated hat structures are bi-interpretable (by the discussion following remark \ref{identification}).
\end{proof}

\begin{ex}
In \cite{MR3047098}, it is shown, by an explicit computation, that the probability algebra $M$ of the unit interval is bi-interpretable with the space $N = L^1(\I, \I)$ of $\I$-valued random variables, identified up to equality almost everywhere. Our reconstruction theorem allows us to recover this result in a more elegant way. Indeed, the probability algebra of $\I$ is separably categorical, thus its automorphism group $G = \text{Aut}(\mu)$ is Roelcke-precompact (\cite[theorem 5.2]{MR2503307}).

Moreover, $G$ is also the automorphism group of $N$. The space of orbit closures of $N$ under the action of $G$ can be identified with the space of probability measures on $\I$. Indeed, given a measurable map in $N$, multiply it by $G$ to make it non-decreasing. The resulting map is then the characteristic function of some probability measure on $I$.

Thus, the space of orbit closures of $N$ is compact. This suffices, by \cite[theorem 2.4]{itaitodor}, to get that the action of $G$ on $N$ is approximately oligomorphic, hence that the structure $N$ is also separably categorical. Theorem \ref{reconstruction} then applies, proving that $M$ and $N$ are bi-interpretable.
\end{ex}

\subsection*{Acknowledgements}
The authors wish to thank Tom\'{a}s Ibarluc\'{i}a for his careful proofreading and the ensuing conversations. Part of this work was carried out during the Universality and homogeneity trimester program at the Hausdorff Institute for Mathematics in Bonn; we are thankful for the extended hospitality and excellent working conditions.

\bibliographystyle{plain}
\bibliography{/Users/adrianekaichouh/Dropbox/Maths/Bibliographie/biblio}

\end{document}